\documentclass{amsart}

\usepackage[utf8]{inputenc}
\usepackage[numbers]{natbib} 
\usepackage{hyperref}
\usepackage{float}
\usepackage{amsmath, amssymb, amsthm}
\usepackage{geometry}
\usepackage{graphicx}
\usepackage{mathrsfs} 
\usepackage{tikz-cd}
\usepackage{mathtools}
\newtheorem{theorem}{Theorem}[section]

\newtheorem{lemma}[theorem]{Lemma}
\newtheorem{corollary}[theorem]{Corollary}
\theoremstyle{definition}
\newtheorem{definition}[theorem]{Definition}
\theoremstyle{remark}
\newtheorem{remark}[theorem]{Remark}
\theoremstyle{plain}
\newtheorem{proposition}[theorem]{Proposition}

\theoremstyle{definition}
\newtheorem{example}{Example}[section]

\DeclareMathOperator{\rH}{\mathrm{H}}
\title{Verdier-Type Category, Geometric Serre's conditions and perversity}
\author{Shang Xu}

\begin{document}

\begin{abstract}
We generalize Serre’s conditions for coherent sheaves to triangulated categories equipped with a \(t\)-structure and duality, which we call Verdier-type categories. We present examples of such categories and their associated Serre conditions, then focus on constructible sheaves, where we refer to these conditions as geometric Serre conditions. We investigate their relationship with perversity and illustrate the theory through a range of examples. Finally, we prove a geometric version of the Auslander–Reiten conjecture. 
\end{abstract}
\maketitle
\tableofcontents
\section{Introduction}
Serre's conditions are fundamental tools in commutative algebra and algebraic geometry. For a Noetherian scheme $X$ and a coherent sheaf $\mathcal F$ on $X$, the condition $S_d$ requires $\operatorname{depth}(\mathcal F_p)\geq \min\{d,\dim\mathcal O_{X,p}\}$ for every $p\in X$. These conditions encode local properties of coherent sheaves and, when applied to the structure sheaf, characterize important geometric features of the underlying scheme. For example, Serre's criterion states that $X$ is normal if and only if it satisfies $S_2$ and $R_1$, the latter being regularity in codimension one. When $X$ is normal, a coherent sheaf satisfies $S_1$ if and only if it is torsion-free, and it satisfies $S_2$ if and only if it is \textit{reflexive}, meaning that the natural map $\mathcal F\to\mathcal F^{\vee\vee}$ is an isomorphism. If $X$ is also integral, the isomorphism classes of rank-one reflexive sheaves correspond naturally to elements of its Weil divisor class group. At the strongest end of this hierarchy, a coherent sheaf satisfying $S_n$ where $n=\dim X$ is called \textit{maximal Cohen--Macaulay}. Such sheaves are closely related to vector bundles: they are locally free wherever $X$ is regular. Indeed, the Auslander--Buchsbaum formula states that $\operatorname{depth}_R M+\operatorname{pd}_R M=\operatorname{depth}R$ for every nonzero finitely generated module $M$ of finite projective dimension over a Noetherian local ring $R$. Consequently, a maximal Cohen--Macaulay module of finite projective dimension is free; in particular, every maximal Cohen--Macaulay module over a regular local ring is free. Thus, the failure of a maximal Cohen--Macaulay sheaf to be locally free reflects singularities of the underlying scheme. Indeed,  \cite{Buchweitz1986} constructed a triangulated singularity category and showed that the maximal Cohen--Macaulay modules form a family of generators.

In Example~\ref{algebraic Serre's conditions}, we use Grothendieck duality to characterize Serre's condition $S_d$ by the inequality
$$
\dim\operatorname{Supp}\mathcal H^i(\mathbb D\mathcal F)\leq -d-i,\qquad i\neq -n,
$$
where $n=\dim X$ and $\mathbb D=R\mathcal{H}om(-,\omega_X)$ is the duality functor associated with the dualizing complex $\omega_X$. This formulation motivates extending Serre's conditions to other sheaf categories equipped with a duality functor. In this paper, we focus on constructible sheaves and call the resulting conditions \textit{geometric Serre's conditions}, which we introduce in \ref{Verdier-type category}. We begin with elementary examples in Section~\ref{first examples}, then investigate their basic properties in Section~\ref{properties}, guided by the classical theory of Serre's conditions for coherent sheaves. In Section~\ref{geometric Serre's conditions and perversity}, we study the relationship between geometric Serre's conditions and perversity. These results connect algebraic and geometric perspectives by interpreting perversity in terms of local depth requirements and Serre's conditions in terms of perversity constraints. More precisely, Proposition~\ref{Generalized duality}, Theorem~\ref{Serre's condition and perversity}, and Theorem~\ref{MCM and intersection complex} show that suitable Serre's conditions imply perversity conditions, while Proposition~\ref{Perverse implies Serre's conditions}, Corollary~\ref{On isolated singularities}, and Corollary~\ref{self dual and finite monodromy} establish converse implications under additional hypotheses. We cite the first two in the introduction since they are the most representative for the connection between our theory and perversity story.
\begin{theorem}[Theorem \ref{Serre's condition and perversity}]
Let $X(r)=\bigcup_{\mathrm{codim_\mathbb C}X_s\leq r}X_s$ to be the union of strata of complex codimension $\leq r$. Let $j_r:U\hookrightarrow X(r)$ be the open inclusion of the union of open strata into $X(r)$. Then if $\mathcal F$ is $S_r$, we have
\begin{enumerate}
    \item [(1)] $\mathcal F[n_\mathbb C]|_{X(r)}$ is perverse on $X(r)$.
    \item [(2)] $\mathcal F[n_\mathbb C]$ is the intermediate extension of $\mathcal F[n_\mathbb C]|_U$ on $X(r-1)$, i.e.,
    $$
    \mathcal F[n_\mathbb C]=(j_{r-1})_{!*}\mathcal F|_U[n_\mathbb C]=\mathrm{IC}(X(r),\mathcal F|_U)
    $$
\end{enumerate}
\end{theorem}

\begin{theorem}[Theorem \ref{MCM and intersection complex}]
Suppose $\mathcal F$ is a constructible sheaf and $\mathcal F|_U=\mathcal L$ is a local system on $U$. Then $\mathcal F$ is maximal Cohen--Macaulay if and only if $\mathcal F=j^\circ_*\mathcal L$ and
$$j^\circ_*\mathcal L[n_\mathbb C]=\mathrm{IC}(X,\mathcal L),\quad j^\circ_*\mathcal L^\vee[n_\mathbb C]=\mathrm{IC}(X,\mathcal L^\vee)
$$
\end{theorem}

Finally, in Section~\ref{Auslander--Reiten Conjecture}, we formulate a geometric analogue of the Auslander--Reiten conjecture and give a complete prove in our case:

\begin{theorem}[Theorem \ref{ARC}]
Suppose $X$ is a equidimensional complex variety with Whitney stratification. Fix a field $k$ of characteristic $0$ and a constructible sheaf $\mathcal F$ on $X$ over $k$, if $\mathcal F$ is geometrically maximal Cohen--Macaulay, and $\mathcal Ext^{>0}(\mathcal F,\mathcal F)=0$, then $\mathcal F$ is a local system. 
\end{theorem}

We then apply this perspective to the classical conjecture, obtaining a geometric reduction of the Gorenstein case to the case of one-dimensional Gorenstein rings. We also present an alternative formulation of the problem that may suggest connections with other areas of mathematics.

The theory of Serre's conditions extends beyond constructible sheaves. The formulation above requires a triangulated category equipped with a $t$-structure defining cohomology objects $H^i$, a duality functor $\mathbb D$, and a dimension filtration on the heart providing a notion of support dimension. Under suitable cohomological vanishing assumptions on $\mathbb D$, many properties of classical Serre's conditions extend to this setting. We formalize these requirements in Definition~\ref{Verdier-type category} and call categories equipped with this structure \textit{Verdier-type categories}. In Section~\ref{Other Verdier-type categories}, we present examples arising from mixed Hodge modules, Nori motives, irregular constructible complexes, and arithmetic $\mathcal D$-modules, as well as the homotopy category of bounded spectra with finitely generated homotopy groups. These examples illustrate the scope of our framework across geometry and homotopy theory. An important source of examples consists of $k$-linear categories equipped with a six-functor formalism and a realization functor to constructible complexes that commutes with the six operations, as in the cases of mixed Hodge modules and Nori motives. These realization functors connect the broader framework to constructible sheaves, highlighting the role of geometric Serre's conditions and motivating our focus on them in this paper. The Serre's conditions in those Verdier categories could be interesting and deserve exploring in the future.

\textbf{Notations.} Throughout this paper, all local systems and constructible sheaves are assumed to be of finite type (i.e., the fibers are finite-dimensional). $k$ usually denotes a field, and whenever there is a finite group actions, we suppose that $\mathrm{char} \, k$ and the group order are coprime. The phrase complex varieties in this paper does not generally refer only to irreducible algebraic varieties over $\mathbb C$, but also to complex algebraic sets which are unions of varieties. 

\textbf{Acknowledgment.} I want to thank my advisor David Eisenbud for inspiring discussions. Thanks also Peisheng Yu, Jerry Yang and Wenqing Wei for useful conversations.

\textbf{AI disclosure.} The ChatGPT 6-Astra is used  to aid in computations of example \ref{intersection complex but not MCM} and \ref{Cn/G}, for literature searches, and for proofreading. The paper was written by the author, who takes full responsibility for its correctness.

\section{Serre's conditions in Verdier-type category}
In this section, we formalize Serre's condition in a category with $t$-structure and duality. We begin by defining Verdier-type category and introducing the main interest of this paper: the geometric Serre's condition. We will also provide other interesting examples of Verdier-type categories in a separate subsection.

\subsection{Verdier-type category and Geometric Serre's conditions}\label{Verdier-type category}
Suppose we are given a triangulated category $\mathcal C$. A \textit{Verdier-type category structure} on $\mathcal C$ consists of the following data:
\begin{enumerate}
\item[$(V_1)$] A t-structure $(\mathcal C^{\geq 0}, \mathcal C^{\leq 0})$, denote by $\mathcal C^\heartsuit$ its abelian heart. 
\item[$(V_2)$] A filtration of Serre subcategories of $\mathcal C^\heartsuit$ 
$$ 0 = \mathcal C^\heartsuit_{-\infty} \subset \mathcal C_0^\heartsuit \subset\mathcal C^\heartsuit_1\subset ...\subset \mathcal C^\heartsuit_n = \mathcal C^\heartsuit
$$
\item[$(V_3)$] A dualizing functor $\mathbb D$:
$$\mathbb D:\mathcal C^{op}\rightarrow\mathcal C
, \quad \iota: \mathbb D\circ\mathbb D\xrightarrow{\sim} id_\mathcal C$$
Satisfying that $\mathbb D\mathcal C^\heartsuit\in \mathcal C^{\leq 0}\cap\mathcal C^{\geq -n}$. 
\end{enumerate}

Denote by ${}^t\rH^i$ the cohomology with respect to the $t$-structure. With these settings, we can formulate the definition of Serre's condition $S_d$ in $\mathcal C$.

\begin{definition}\label{Serre's condition}
Suppose $A$ is an object in $C^\heartsuit$. For an integer $0\leq d\leq n$, we say $A$ satisfies \textit{Serre's condition $S_d$} if for $i\neq -n$, we have
$$\quad {}^t\rH^i(\mathbb DA)\in \mathcal C^\heartsuit_{-i-d} 
$$
where $\mathcal C^\heartsuit_r=\mathcal C^\heartsuit_{-\infty}=0$ if $r<0$ and $\mathcal  C^\heartsuit_r=\mathcal C^\heartsuit$ if $r>n$.
\end{definition}
As an extension of classical language, if a object satisfying $S_n$, we call it a \textit{maximal Cohen--Macaulay} (MCM) object.
\begin{example}\label{algebraic Serre's conditions}
For an $n$-dimensional Cohen--Macaulay variety $X$ with dualizing sheaf $\omega_X$: Take $\mathcal C = D^b(\mathfrak{Coh}(X))$, t-structure $(D^b_+(\mathfrak{Coh}(X)), D^b_-(\mathfrak{Coh}(X)))$; Take dimension filtration $\mathcal C_i,\ i=\{-\infty\}\cup\mathbb Z_{[0,n]}$ to be the full subcategory of $\mathfrak{Coh}(X)$ consisting of sheaves with dimension of support $\leq i$; and finally take dualizing functor $\mathbb D=R\mathcal{H}om(-,\omega_X[n])$.. Then this definition \ref{Serre's condition} actually recovers the classical Serre's condition on coherent sheaves, as shown in the following.

Suppose $\mathcal F$ is a coherent sheaf on $X$. The Serre's condition $S_d$ requires for each point $p\in X$, one has
$$\mathrm{depth}(\mathcal F_p)\geq \min\{d,\dim\mathcal O_{X,p}\}
$$
Using local cohomology, this is the same as requiring $\rH^i_p(\mathcal F_p)=0$ for $i<\min\{d,\dim\mathcal O_{X,p}\}$. By Grothendieck local duality, this is equivalent to 
$$\mathrm{Ext}_{\mathcal{O}_{X,p}}^i(\mathcal F_p,\omega_{X,p})=0 \ \ \text{for}\ i> \dim\mathcal O_{X,p}-\min\{d,\dim\mathcal O_{X,p}\}
$$
Since sheaf Ext localizes, this is also
$$\mathcal{E}xt_{\mathcal{O}_{X}}^i(\mathcal F, \omega)_p=0 \ \ \text{if} \ \begin{cases} d\geq\dim\mathcal O_{X,p}, \  i>0\\
d <\dim \mathcal{O}_{X,p}, \ i>\dim\mathcal O_{X,p} - d
\end{cases}
$$
Now $\dim\mathcal{O}_{X,p}=n-\dim\overline{\{p\}}$, so the two conditions above are $\dim\overline{\{p\}}\geq n-d$ and $n-d-i<\dim\overline{\{p\}}<n-d$, i.e., the $S_d$ condition translates into
$$\mathcal{E}xt_{\mathcal{O}_{X}}^i(\mathcal F, \omega)_p=0 \ \ \text{if} \ \dim\overline{\{p\}}>n-d-i\  \text{and} \ i>0
$$
Shifting by the dimension of $X$, we see that this is also equivalent to
$$\dim\mathrm{Supp} \mathcal H^i(\mathbb D\mathcal{F})\leq -d-i,\ i\neq -n
$$
Notice that $i\neq -n$ is the same as $i>-n$ here, since when $i<-n$, $\rH^i(\mathbb D\mathcal F)$ is automatically $0$.
\end{example}

\begin{example}
Suppose $k$ is field. For a reasonably stratified locally compact Hausdorff space $X$ (for example, a complex variety with a Whitney stratification) of finite cohomology dimension, we have Verdier dualizing functor $\mathbb D=R\mathcal{H}om(-,a_X^! k)$ on the derived category of constructible sheaves $\mathcal C=D_c^b(X)$. Again we take the filtration as support-of-dimension filtration, the $t$-structure as the usual (not the perverse) $t$-structure. When $X$ is a locally closed subset of $\mathbb R^N$, the definition of $S_d$ condition is simply
$$\dim_\mathbb R \mathrm{Supp} \mathcal H^i(\mathbb D\mathcal F)\leq -i-d,\quad i\neq -\dim_\mathbb R X
$$
In our convention, if we fixed a good stratification $\mathcal S$, the constructible sheaf we consider will be those constructible with respect to $\mathcal S$.
\end{example}

We call this \textit{geometric Serre's conditions}. This is what we will focus in the following chapters. We will see how these conditions similar to algebraic Serre's conditions an the relation between those conditions and perversity. Many interesting examples will be given.

When $k$ is any commutative noetherian ring with finite global dimension, it also make sense to define Serre's condition $S_d$ as
$$ \mathrm{mdSupp} \mathcal H^i(\mathbb D\mathcal F)\leq -i-d,\quad i\neq -\dim_\mathbb RX
$$
Here $\mathrm{mdSupp}$ is the modified dimension of support introduced in \cite{Achar2021PerverseSheaves}. Unfortunately, under this setting, the cohomology bound does not satisfy $(V_3)$, so further considerations are needed here.

\subsection{Other Verdier-type categories}\label{Other Verdier-type categories}

We will see more interesting examples of Verdier-type categories in this subsection. Those reveal the generality of our theory and explain why it stimulate our interest.
\begin{example}
Fix a subfield $k$ of $\mathbb R$, suppose $X$ is a complex variety and let $\mathcal C=D^b\mathrm{MHM}(X,k)$, the bounded derived category of mixed Hodge modules. Take the $t$-structure to be the constructible $t$-structure:
$${}^{ct}\mathcal C^{\leq 0}=\{M:\mathrm{rat}M\in D_c^{\leq 0}(X)\},\quad {}^{ct}\mathcal C^{\geq 0}=\{M:\mathrm{rat}M\in D_c^{\geq 0}(X)\}
$$
Now take $\mathbb D=R\mathcal{H}om(-,a_X^! k(0))$ where $a_X: X\rightarrow \mathrm{Spec}\ \mathbb C$. By \cite{Tubach2025MixedHodgeStacks} Proposition 4.2.3., $\mathbb D$ is a dualizing functor. The commutativity of rat with the six functors implies that these settings provide a Verdier-type category structure on $\mathcal C$. The $S_d$ condition on a complex $\mathcal M$ in $\mathcal C$ is interpreted as
$$\dim_\mathbb R\mathrm{Supp}\ \mathrm{rat}{}^{ct}\mathcal H^i(\mathbb D\mathcal M)\leq -i-d.
$$
Which is the same as $\mathrm{rat} \mathcal M\in D^b_c(X)$ satisfies the geometric Serre's condition $S_d$.
\end{example}

More generally, as explained in \cite{Tubach2025NoriHodge}, suppose $k$ is a field and $\mathcal C:\mathrm{Var}^{op}\rightarrow \mathrm{Triang}$ is a $k$-linear triangulated category of coefficients with a  six-functor formalism and a
conservative realisation functor $R : \mathcal C \rightarrow D^b_c$ compatible with the operations. Define the constructible $t$-structure as
$${}^{ct}\mathcal C^{\leq 0}=\{M\in\mathcal C|\ x^*M\in D^{\leq 0}(\mathrm{Spec} \,\kappa(X))\},\quad {}^{ct}\mathcal C^{\geq 0}=\{M\in\mathcal C|\ x^*M\in D^{\geq 0}(\mathrm{Spec} \,\kappa(X))\}
$$
Then as long as the Verdier functor $\mathbb D$ is an equivalence of categories, we have a Verdier-type category structure. For example, for $\mathcal C=\mathcal M(X)$, the category of perverse motives. By \cite{IvorraMorel2024FourOperations} Proposition 2.6, the corresponding Verdier functor $\mathbb D_X^\mathcal M$ is an equivalence of categories. Another example is the derived category of irregular constructible complexes $D^b_{ic}(\Lambda_X)$ introduced in \cite{Kuwagaki2021IrregularPerverse}, in which the properties of the Verdier dualizing functor follows from Corollary 5.38 and Lemma 9.4.
\begin{example}
Suppose $\mathfrak{P}$ is a smooth formal scheme with special fiber of dimension \(d\). \cite{Lazda2023Rigidification} introduced a triangulated category $\mathcal C=D_{hol,F}^b(\mathfrak{P})$, the the full subcategory on objects in the bounded derived category of overholonomic $\mathcal D^\dag_{\mathfrak{P}\mathbb Q}$ modules which are of Frobenius type, where $\mathcal D^\dag_{\mathfrak{P}\mathbb Q}$ is the Berthelot’s ring of overconvergent differential operators of $\mathfrak{P}$. The constructible \(t\)-structure is defined in $\S$ 4.4.2 of \cite{Lazda2023Rigidification}, and the dualizing functor is shown in \cite{Caro2009Overholonomic} (specifically, equation (1.10.6) and equation (1.10.7)). The cohomology range of dualizing functor can be seen by combining the formulas in \cite{Lazda2023Rigidification} $\S$4.4.2 with self-duality of the holonomic \(t\)-structure in $\S$ 4.4.1.
\end{example}

Let's see one more example, where we can explicitly write down the meaning of each Serre's condition.
\begin{example}
Take \(\mathcal C=\operatorname{Ho}(\mathrm{Sp}^b_{\mathrm{fg}})\), the bounded spectra with finitely generated homotopy groups. \cite{HebestreitLandNikolaus2021LSpectra} Proposition 3.2.7 states the Postnikov \(t\)-structure and identifies its heart with abelian groups. it refers to \cite{Lurie2017HigherAlgebra}, Proposition 1.4.3.6. Restricting to bounded spectra with finitely generated homotopy groups gives the heart $\mathrm{Ab}_{\mathrm{fg}}$. More concretely, our setting is 
\[
\begin{aligned}
\mathcal C^{\le0}
&=\{E:\pi_j(E)=0\text{ for }j<0\},\\
\mathcal C^{\ge0}
&=\{E:\pi_j(E)=0\text{ for }j>0\}.
\end{aligned}
\]
and its heart is 
$$
\mathcal C^\heartsuit=\{E:\pi_j(E)=0\text{ for }j\ne0\}=\{HA|A\in\mathrm{Ab_{fg}}\}\cong\mathrm{Ab_{fg}}
$$
where $HA$ is the Eilenberg–MacLane spectrum with respect to the finite generated abelian group $A$.

In \cite{HebestreitLandNikolaus2021LSpectra}, Definition 2.2 defined Anderson duality \(I\) as follows: For an injective abelian group \(J\), let \(I_J(\mathbb S)\) be the spectrum representing \([E,I_J]\cong\operatorname{Hom}_{\mathbb Z}(\pi_0E,J)\) and let $I_J(X)=\mathrm{map}(X,I_J(\mathbb S))$ for any spectra $X$. The Anderson dual of spectrum $X$ is the one that fits into the fiber sequence $I(X)\rightarrow
I_{\mathbb Q}(X)\rightarrow I_{\mathbb Q/\mathbb Z}(X)$.
One has \(I(X)\simeq F(X,I(\mathbb S))\), and  Theorem 2.9 shows that $I\circ I =id$. Now we can define our dualizing functor as
$$\mathbb D(X)=I(X)[1]=\mathrm{map}(X,I(\mathbb S)[1])
$$

Now \cite{HebestreitLandNikolaus2021LSpectra} Example 2.5 tells us that $\mathbb DH\mathbb Z=H\mathbb Z[1]$, $\mathbb DH\mathbb Z/n=H\mathbb Z/n$ for positive integer $n$, hence $\mathbb D HA\in \mathcal C^{\leq 0}\cap \mathcal C^{\geq -1}$ for any finitely generated abelian group $A$. Therefore, we just need a one step dimension filtration: let $\mathrm{Ab}_{\mathrm{fin}}$ to be the finite abelian groups and take
$$
0\subset\mathrm{Ab}_{\mathrm{fin}}\subset\mathrm{Ab}_{\mathrm{fg}}
$$ 
to be the dimension filtration $0=\mathcal C_{-\infty}^\heartsuit\subset\mathcal C_0^\heartsuit\subset\mathcal C_1^\heartsuit=\mathrm{Ab}_{\mathrm{fg}}$. Now for a finitely generated abelian group $A=\mathbb Z^m\oplus T$ where $T$ is torsion, we see that 
$${}^t\rH^i(\mathbb D HA)= \begin{cases}
    HT, \quad i=0\\
    H\mathbb Z^m \quad i=-1\\
    0 \quad\quad \text{otherwise}
\end{cases}
$$
Now $S_0$ is requiring $HT\in\mathrm{Ab}_{\mathrm{fin}}$, which automatically holds; $S_1$ is same as requiring $HT=0$, so $T=0$, i.e., $A$ is torsion free. So we conclude
Thus
$$
HA\text{ satisfies }S_1
\iff A\text{ is torsion-free}
\iff A\simeq\mathbb Z^m,
$$
This is the maximal Cohen--Macaulay condition in this category.
\end{example}

\section{Geometric Serre's conditions and Perversity}
The various example from last section with a realization functor illustrates the importance of Verdier-type category structure on constructible sheaves, so it deserves a first exploitation.

In this section, we suppose $X$ is a reasonably stratified real space, i.e., $X$ is an algebraic/analytic space/CW complex with a stratification such that shriek restriction/extension preserves constructible strata and there is Verdier duality. Complex varieties with Whitney stratification are the most common settings, and the reader are encouraged to assume that. 

We also suppose $X$ is $n$-\textit{equidimensional}, i.e., there is a dense open set $U$ in $X$ which is a $n$-dimensional manifold. One may also suppose this open set is a union of open strata. We will follow the notations in \cite{Achar2021PerverseSheaves}, denote by $f_*$, $f_!$, $f^!$ the derived functors, and use $f_*^\circ$, $f_!^\circ$ to denote the nonderived functor if there is difference.

\subsection{First examples}\label{first examples}
Let's recall our definition of Serre's condition $S_d$ for constructible sheaf $\mathcal F$ on $X$:
$$\dim_\mathbb R\mathrm{Supp}\mathcal H^i(\mathbb D\mathcal F)\leq -i-d,\quad i\neq -n
$$
For $d\leq 0$, this automatically holds for any constructible sheaf. We are hence only interested in $S_d$ conditions with $d>0$, especially the maximal Cohen--Macaulay (MCM) condition $S_n$, which can be interpreted as
$$\mathbb D\mathcal F =\mathcal G[n] \ \text{for some constructible sheaf} \ \mathcal G
$$
Some examples help us understand this condition.

\begin{example}
Suppose $X$ is a smooth variety and take itself as a stratification. The constructible heart consists of local systems on $X$. The dualizing functor is the non-derived dual shifted by $\dim_\mathbb R X$, so for any local system $\mathcal{L}$, we have $\mathbb D\mathcal L=\mathcal L^\vee[\dim_\mathbb R X]$. Hence any local system is MCM in our case
\end{example}
\begin{example}
Let $X=[0,1)\subset\mathbb R$, take the stratification $X=\{0\}\cup(0,1)$, and let $k$ be a field. A sheaf $\mathcal F$ constructible with respect to this stratification is characterized by two finite-dimensional vector spaces $V_0$ and $V_{(0,1)}$ and a map $\mathrm{res}:V_0\rightarrow V_{(0,1)}$. Let $j:(0,1)\hookrightarrow [0,1)$; the dualizing complex of $X$ is $j_!k_{(0,1)}[1]$. For $0<\varepsilon<1$, we have 
$$
R\Gamma_c([0,\varepsilon),\mathcal F)\simeq[V_0\xrightarrow{\mathrm{res}}V_{(0,1)}]
$$
with the two terms in degrees $0$ and $1$. By Verdier duality, $\mathbb D_X\mathcal F$ is concentrated in degrees $-1$ and $0$, and $\mathcal H^0(\mathbb D_X\mathcal F)$ is the skyscraper sheaf at $0$ with stalk $(\ker\mathrm{res})^\vee$. Since $\dim_{\mathbb R}X=1$, the MCM condition is equivalent to $\mathcal H^0(\mathbb D_X\mathcal F)=0$, hence to $\mathrm{res}$ being injective. Choosing a complement to $\mathrm{res}(V_0)$ in $V_{(0,1)}$, we obtain a noncanonical decomposition $\mathcal F\simeq k_X^{\oplus a}\oplus(j_!k_{(0,1)})^{\oplus b}$, where $a=\dim_kV_0$ and $b=\dim_kV_{(0,1)}-\dim_kV_0$. Thus $k_X$ and $j_!k_{(0,1)}$ are precisely the two indecomposable MCM sheaves constructible with respect to this stratification.

There is another simpler way to see: By the equivalent description of MCM condition, $S_1$ is the same as requiring $\mathbb D\mathcal F=\mathcal G[1]$ for another constructible sheaf $\mathcal G$. Since $\mathcal F$ is already a local system on $\mathbb C^\times$, it's the same as requiring $\mathbb D\mathcal F|_0$ concentrated in degree -1. Now by Verdier duality: $\mathbb D\mathcal F|_0=i_0^*\mathbb D\mathcal F=\mathbb Di^!\mathcal F$. Now take the dual of the triangle
$$i_0^!\mathcal F\rightarrow i_0^*\mathcal F\rightarrow i_0^*j_*\mathcal F\xrightarrow{+1}
$$
We get $\rH^0(i^!\mathcal F)=0$ if and only if $\mathrm{res}: \rH^0(i^*_0(F))=V_0\rightarrow V_{(0,1)}=\rH^0(i_0^*j_*\mathcal F)$ is injective, which is the conclusion above.
\end{example}

Like in the algebraic case, one may expect MCM sheaves on smooth part to be locally constant, however, the following example shows that if codimension $1$ strata exists, this can fail.

\begin{example}\label{punctured disck}
Take $X=\mathbb C$, we stratify it by $\mathcal S:\mathbb C=\{0\}\cup\mathbb C^\times$ and let $i_0:0\rightarrow\mathbb C$, $j:\mathbb C^\times\rightarrow\mathbb C$. The constructible sheaf with respect to this strata can be characterized by a local system $\mathcal L$ on $\mathbb C^\times$, a vector space $V_0$, and a restriction map $\mathrm{res}:V_0\rightarrow  i_0^*j_*^\circ \mathcal L$. As $S_d$ conditions are local, it suffices to check the story at $0$. The $S_2$ condition is the same as requiring $\mathbb D\mathcal F$ to be concentrated in degree $-2$, which, by Verdier duality, is also the same as $i_0^!\mathcal F$ to be concentrated in degree $2$. Again take the dual of the triangle $i_0^!\mathcal F\rightarrow i_0^*\mathcal F\rightarrow i_0^*j_*\mathcal F\xrightarrow{+1}$, we see that this is equivalent to $\mathrm{res}: \rH^0(i^*_0(F))=V_0\rightarrow i_0^*j_*^\circ \mathcal L=\rH^0(i_0^*j_*\mathcal F)$ being an isomorphism. Similarly, $S_1$ is equivalent to $\mathrm{res}$ being injective. 

As a consequence, for any local system $\mathcal{L}$ on $\mathbb C^\times$, the pushforward $j_*^\circ\mathcal L$ is MCM. The set of simple objects is $\{j_*^\circ\mathcal L_V\}_{V\in\mathrm{Rep}_k\mathbb Z}$, where $\mathcal L_V$ is the local system corresponding to a $\mathbb Z$-representation $V$. Since $\mathrm{Hom}(j_*^\circ\mathcal L_V,j_*^\circ\mathcal L_W)=\mathrm{Hom}(j^*j_*^\circ\mathcal L_V, \mathcal L_W)=\mathrm{Hom}(\mathcal L_V,\mathcal L_W)$, so $\mathrm{MCM}(X,\mathcal S)\cong \mathrm{Loc}(\mathbb C^\times)\cong\mathrm{Rep}_k\mathbb Z$.

However, for $n\geq 2$, strata $\mathcal S:X=\mathbb C^n =\{0\}\cup(\mathbb C^n-\{0\})$, $\mathrm{MCM}(X,\mathcal S)$ contains only constant sheaf since $\pi_1(\mathbb C^n-\{0\})$ is trivial.
\end{example}

One may expect the constant sheaf to be always MCM, but this is also not always true, by looking at the following example.
\begin{example}\label{sub nilpotent orbit}
Suppose $X=\mathcal N_{3,1}$, the variety of $3\times 3$ nilpotent matrices of rank $\leq 1$. It has complex dimension $4$, hence real dimension $8$, and whose only singular point is $0$. Take the strata $\mathcal S: X=\{0\}\cup U$. Since $X$ has resolution $T^*\mathbb P^2$, so $U=T^*\mathbb P^2-s(\mathbb P^2)$, where $s$ is the zero section. The Gysin sequence tells us $\rH^3(U,\mathbb C)=H^2(\mathbb P^2,\mathbb C)=\mathbb C$. Since $X$ is contractible, $\rH^2(\mathbb C_X)=\rH^3(\mathbb C_X)=0$. Then
$$
\rH^3(i_0^!\mathbb C_X)\cong \rH^3_{\{0\}}(\mathbb C_X)\cong H^2(U,\mathbb C)=\mathbb C.
$$
But the MCM condition together with Verdier duality requires $i_0^!\mathbb C_X$ to be concentrated on degree $4$, so $\mathbb C_X$ is not maximal Cohen--Macaulay. Further computation shows that $\mathbb C_X$ is $S_3$.
\end{example}

\begin{remark}
As we will show in the next section, like in coherent case, if a constructible sheaf $\mathcal F$ is $S_2$, then $\mathcal F=j^\circ_*j^*\mathcal{F}$ for any open set $j: U\hookrightarrow X$ which is a union of strata such that $\mathrm{codim}_\mathbb R(X-U)\geq 2$. So in this example, $\mathbb C_X$ is the only $S_2$ sheaf with respect to this strata since $\pi_1(U)=\pi_1(T^*\mathbb P^2-s(\mathbb P^2))=1$. 

One may think of the locally constant sheaves to be the most possibly MCM, but look at this example $X=\mathbb C[x,y,z,w]/(xz,xw,yz,yw)$, $X$ is the union of two planes meeting at $0$. Since $H^1(i^!_0\mathbb C_X)= \mathrm{coker}\ \mathbb C \xrightarrow{\Delta} \mathbb C^2 = \mathbb C$, so $\mathbb C_X$ is not $S_2$ (it is $S_1$). However, $j_*^\circ\mathbb C_{X-\{0\}}$ is maximal Cohen--Macaulay.
\end{remark}

Let's see one more interesting example. This example motivates this whole paper and illustrates the relationship between geometric and algebraic MCM families.

\begin{example}\label{Kleinian surfaces}
Take $X=\mathbb C^2/G$, $G$ is a finite group in $SL(2,\mathbb C)$. These are called the \textit{Kleinian surfaces} with only singularity of $X$ is at $\overline 0$. Such singularities are known as \textit{ADE singularities} or \textit{rational double points}. Now let's stratify $X$ by $\mathcal S: X=\{\overline 0\}\cup U$. Choose a $G$-invariant Hermitian metric on $\mathbb C^2$, we see that $U$ is homotopy equivalent to $S^3/G$. Suppose our base field $k$ satifisies $\mathrm{char }\, k\nmid |G|$. For any constructible sheaf $\mathcal F$ over $k$, suppose $\mathcal F|_U$ corresponds to $G$-representation $V$. Again consider the restriction of localization triangle, 
\[
i^!\mathcal F\longrightarrow \mathcal F_0
\longrightarrow R\Gamma(S^3/G,\mathcal L_V)
\xrightarrow{+1}
\]
Let $a : \mathcal F_0\rightarrow i^*j^\circ_*\mathcal F|_U$. Since $G$ is finite, $H^q(S^3/G,\mathcal L_V)\cong H^q(S^3, V)^G$, and this is $0$ if $q\neq 3$ and $V^G$ if $q=3$, there is
\[
H^q(i^!\mathcal F)=
\begin{cases}
\ker(a),&q=0,\\
\operatorname{coker}(a),&q=1,\\
V^G,&q=4,\\
0,&\text{otherwise}.
\end{cases}
\]
Hence the MCM constructible sheaves are exactly those $\mathcal F=j_*\mathcal L_V$ for some $G$-representation $V$, hence
$$\mathrm{MCM}(X,\mathcal S)\cong\mathrm{Loc}_k(U)\cong\mathrm{Rep}_k G
$$
What is interesting here is, this is also the classification of algebraic maximal Cohen--Macaulay modules. \cite{ArtinVerdierEsnault1985} classified reflexive modules over rational double point by lifting the modules to the minimal resolution and consider the intersection of the their first Chern class and the exceptional curves. They proved that the indecomposable reflexive modules lift to vector bundles whose first Chern class intersects exactly one single $\mathbb P^1$ in the exceptional fiber, and the intersection number reflects the multiplicity. According to McKay correspondence, the components in the exceptional locus are one-to-one corresponds to the vertices of the McKay-Springer graph of the natural representation of $G$, which is also one-to-one corresponding to irreducible representations of $G$. While we start from an irreducible representation $V$, define the associated local system $\mathcal L_V$ on $\mathbb C^2/G$, and form the corresponding vector bundle $\mathcal E_V$. The localization of the reflexive sheaf $j_*\mathcal E_V$ at $0$ is exactly the irreducible reflexive sheaf corresponding to $V$ along the lifting-Chern-class-intersection-McKay-correspondence path. This can be found in \cite{GSV1983} 2.7-2.10. This correspondence tells the naturality of our definition, generalizing the classification of MCM modules over rational double points in a computable manner.
\end{example}

\subsection{Basic properties, torsion freeness and reflexivity}\label{properties}

In this subsection, we follow the properties of algebraic Serre's conditions to characterize geometric Serre's conditions. We continue our convention of considering a reasonably stratified $n$-equidimensional real space. 

First, we have a lemma that generalize depth inequality in commutative algebra. This actually holds for generalized Serre's condition in Verdier-type categories, so we state the most general form here.
\begin{lemma}\label{Serre's conditions in extensions}
Suppose $\mathcal C$ is a Verdier-type category and there is an extension of objects in $\mathcal C^\heartsuit$
$$0\rightarrow\mathcal E\rightarrow \mathcal F\rightarrow\mathcal G\rightarrow 0
$$
Then we have
\begin{itemize}
    \item [(1)] If $\mathcal E$ and $\mathcal G$ are $S_d$, then so is $\mathcal F$.
    \item [(2)] If $\mathcal F$ is $S_d$ and $\mathcal G$ is $S_{d-1}$, then $\mathcal E$ is $S_d$.
    \item [(3)] If the extension splits, then $\mathcal F$ is $S_d$ if and only if $\mathcal E$ and $\mathcal G$ is $S_d$.
\end{itemize}
\end{lemma}
\begin{proof}
This can be easily seen by looking at the long exact cohomology sequence of the dual triangle 
$\mathbb D\mathcal G\rightarrow\mathbb D\mathcal F\rightarrow\mathbb D\mathcal E\xrightarrow{+1}$.
\end{proof}

Next, we wish to recover the classical form of algebraic Serre's condition, as in Example \ref{algebraic Serre's conditions}, so we can better understand the geometric conditions locally. We have the following result:
\begin{proposition}\label{Generalized duality}
Suppose $X$ is $n$-equidimensional space with a reasonable (for example, Whitney) stratification $\mathcal S: X=\bigcup_{s}X_s$, then for a constructible sheaf $\mathcal F$ along $S$, the following two conditions are equivalent.
\begin{itemize}
    \item[(1)] $\mathcal F$ is $S_d$.
    \item[(2)] For any strata $X_s$, we have $i_{X_s}^!\mathcal F \in D^{\geq\min\{d,\mathrm{codim}_\mathbb R X_s\}}_{\mathrm{Loc}}(X_s)$
\end{itemize}
\end{proposition}
\begin{proof}
First we show (1) implies (2). Suppose $\mathcal F$ is $S_d$, i.e., $\dim_\mathbb R\mathrm{Supp}\mathcal H^i(\mathbb D\mathcal F)\leq -i-d$ for $i\neq -n$, where $n=\dim_\mathbb R X$. Since we suppose the stratification is good, $\mathbb D\mathcal F$ is locally constant along all strata. So if $\mathcal H^i(\mathbb D\mathcal F)$ is nonzero along some $X_s$, we have $X_s\subseteq\mathrm{Supp}\mathcal H^i(\mathbb D\mathcal F)$. Hence $i_{X_s}^*\mathcal H^i(\mathbb D\mathcal F)=0$ if $\dim_\mathbb R X_s\leq -i-d$ and $i\neq -n$. While this condition is the same as $i>-\dim_\mathbb R X_s-d$, $i\neq -n$, we conclude that $i_{X_s}^*\mathbb D\mathcal F\in D^{\leq\max\{-\dim_\mathbb R X_s-d,-n\}}_{\mathrm{Loc}}(X_s)$, hence
$$i_{X_s}^!\mathcal F=\mathbb D\mathbb D(i^!_{X_S}\mathcal F)=\mathbb D(i_{X_s}^*\mathbb D\mathcal F)=\mathrm{Hom_{Loc}}(i_{X_s}^*\mathbb D\mathcal F, k[\dim_\mathbb R X_s])\in D^{\geq\min\{d,\mathrm{codim}_\mathbb R X_s\}}_{\mathrm{Loc}}(X_s)
$$

Conversely, suppose (2) holds. From the discussion above, we know that (2) implies $i_{X_s}^*\mathbb D\mathcal F\in D^{\leq\max\{-\dim_\mathbb R X_s-d,-n\}}_{\mathrm{Loc}}(X_s)$, so $\mathcal H^i(i_{X_s}^*\mathbb D\mathcal F)\neq0$ either for $i\leq-\dim_\mathbb R X_s-d$, i.e., $\dim_\mathbb R X_s\leq-i-d$, or for $\max\{-\dim_\mathbb R X_s-d,-n\}=-n$ and $i\leq -n$. But $\mathbb D\mathcal F\in D^{\geq -n}(X)$, so in this case $i=-n$. This implies $\dim_\mathbb R\mathrm{Supp}\mathcal H^i(\mathbb D\mathcal F)\leq -i-d$ for $i\neq -n$, so $\mathcal F$ is $S_d$.
\end{proof}

Making use of this, we can prove many results that are is the analogue of algebraic Serre's conditions. First, for a constructible sheaf $\mathcal F$, we say $\mathcal F$ is \textit{torsion free} if $\mathcal F$ contains no sections supported on any proper closed union of strata whose complement is dense. Suppose $j: U\rightarrow X$ is the union of $n$-dimensional open strata in $X$, we use the notation $\mathcal F^\vee=\mathcal Hom(\mathcal F, j_*^\circ k_U)$

\begin{proposition}\label{torsion free and S1}
Under the assumption of \ref{Generalized duality}, the following conditions are equivalent:
\begin{itemize}
    \item[(1)] $\mathcal F$ is $S_1$.
    \item[(2)] $\mathcal F$ is torsion free.
    \item[(3)] The natural map $\mathcal F\rightarrow\mathcal F^{\vee\vee}$ is injective.
    \item[(4)] $\mathcal F$ is a subsheaf of a torsion free sheaf.
\end{itemize}
\end{proposition}
\begin{proof}
By \ref{Generalized duality}, $S_1$ is equivalent to $i_{X_s}^!\mathcal F\in D^{\geq 1}(X_s)$ for stratum $X_s$ with codimension $\geq 1$. Now for any closed union of strata $Z$, the strata in $Z$ should all of codimension $\geq 1$. One can prove by noetherian induction that $i_Z^!\mathcal F\in D^{\geq 1}(Z)$ using the localization sequence. For example, if there are only two strata $i:W\hookrightarrow Z$ and $j:V\hookrightarrow Z$ with $W$ closed, let $\mathcal G=i_Z^!\mathcal F$, we have
$$i_*i^!\mathcal G\rightarrow\mathcal G\rightarrow j_*j^*\mathcal G\xrightarrow{+1}
$$
Where $j^*=j^!$, so $j^!\mathcal G$ and $i^!\mathcal G$ are in $D^{\geq 1}$ implies $\mathcal G\in D^{\geq 1}$. Then according to the localization triangle
$${i_Z}_*i^!_Z\mathcal F\rightarrow\mathcal F\rightarrow j_*j^*\mathcal F\xrightarrow{+1}
$$
we see that $\mathcal F\rightarrow j_*^\circ j^*\mathcal F$ is injective. In particular, there is no section supported on $Z$ (or it's in the kernel of this map), so (1) implies (2). Conversely, (2) tells us $\rH^0(i^!_Z(\mathcal F))$ vanishes for any proper closed union of strata whose complement is dense, then so does $\rH^0(i_{X_s}^!\mathcal F)$ for any stratum $X_s$ of codimension $\geq 1$ since $i_{X_s}^!$ is left exact, and $\overline{X_s}$ is a closed union of strata whose complement is dense (remember we assume $X$ is equidimensional. So (1) and (2) are equivalent.

Now (2) implies (3) because the sections in the kernel of $\mathcal F\rightarrow\mathcal F^{\vee\vee}$ is supported on $X-U$ where $U$ is the dense open union of stratum, while (2) tells you that there is no such section.

(3) implies (4) because the dual of any constructible sheaf $\mathcal G$ is torsion free: $\mathcal G^\vee=\mathcal Hom(\mathcal G,j_*^\circ k_U)=j_*\mathcal Hom(\mathcal G|_U,k_U)$. Now $\mathcal G|_U$ is a local system, then so is $\mathcal Hom(\mathcal G|_U,k_U)$. A section of $j_*\mathcal Hom(\mathcal G|_U,k_U)$ is the same as a section in $\mathcal Hom(\mathcal G|_U,k_U)$, and it vanishes on a dense open set if and only if it is zero, so $\mathcal G^\vee$ is torsion free. In particular, $\mathcal F^{\vee\vee}$ is torsion free, so (3) implies (4).

Finally, (4) implies (2) because $\rH^0_Z=\Gamma_Z$ is left exact. Therefore, the four conditions are equivalent.
\end{proof}

Next, we can make sense of reflexivity and relate it to $S_2$. A constructible sheaf $\mathcal F$ is said to be \textit{reflexive} if $\mathcal F\rightarrow\mathcal F^{\vee\vee}$ is an isomorphism. 
\begin{proposition}\label{reflexive}
Suppose $X$ is a complex variety with Whitney stratification and $j:U\rightarrow X$ is the union of open strata. If $k_X$ is $S_2$, then the following conditions are equivalent:
\begin{itemize}
    \item [(1)]$\mathcal F$ is $S_2$.
    \item [(2)]$\mathcal F =j_*^\circ\mathcal L$ for some local system $\mathcal L$ on $U$.
    \item [(3)]$\mathcal F$ is reflexive.
    \item [(4)]$\mathcal F=\mathcal G^\vee$ for some constructible sheaf $\mathcal G$.
    \item [(5)]$\mathcal F=\mathcal Hom(\mathcal G,\mathcal R)$ for some reflexive sheaf $\mathcal R$.
\end{itemize}
\end{proposition}
\begin{proof}
Suppose (1) holds, $\mathcal F$ is $S_2$, then $i_{X_s}^!\mathcal F\in D^{\geq 2}(X_s)$ for any non-open strata $X_s$ (since they are of real codimension $\geq 2$). Let $Z=X-U$, one can prove by noetherian induction as above that $i_Z^!\mathcal F\in D^{\geq 2}(Z)$.  Then take the cohomology of the triangle ${i_Z}_*i_Z^!\mathcal F\rightarrow\mathcal F\rightarrow j_*j^*\mathcal F\xrightarrow{+1}$, we have $\mathcal F\xrightarrow{\sim}j_*^\circ j^*\mathcal F$, and $j^*\mathcal F$ is a local system $\mathcal L$, so (2) holds.

Conversely, if (2) holds, then the localization triangle tells you $i_Z^!\mathcal F\in D^{\geq2}(Z)$, then $i_{X_s}^!\mathcal F\in D^{\geq2}(X_s)$ for each stratum $X_s$ since $i^!$ is left exact, so $\mathcal F$ is $S_2$, hence (1) and (2) are equivalent.

Now if $\mathcal F$ is $S_2$, a fortiori, it is $S_1$, so by \ref{torsion free and S1}, $\mathcal F\rightarrow \mathcal F^{\vee\vee}$ is injective. Now the support of the quotient $\mathcal Q$ is contained in $Z$, but looking at the triangle $i_Z^!\mathcal F\rightarrow i_Z^!\mathcal F^{\vee\vee}\rightarrow i_Z^!\mathcal Q\xrightarrow{+1}$, we see that $\mathcal H^0_Z(\mathcal Q)=0$. This tells you $\mathcal Q=0$, so $\mathcal F$ is reflexive. Then (1) implies (3). Now (3) implies 
$$\mathcal F=\mathcal F^{\vee\vee}=\mathcal Hom(\mathcal F^\vee ,j^\vee_*k_U)=j_*\mathcal Hom(\mathcal F^\vee|_U,k_U)
$$
and $\mathcal Hom(\mathcal F^\vee|_U,k_U)$ is a local system on $U$. So (3) implies (2). Then (1),(2),(3) are equivalent.

Now (3) obviously implies (4), (4) implies (5) because $j_*^\circ k_U$ is reflexive:
\begin{align}\notag
(j_*^\circ k_U)^{\vee\vee}&=\mathcal Hom(\mathcal Hom(j_*^\circ k_U,j_*^\circ k_U),j_*^\circ k_U)\\\notag
&=\mathcal Hom(j_*\mathcal Hom(k_U,k_U),j_*^\circ k_U)=\mathcal Hom(j_*^\circ k_U,j_*^\circ k_U)=j_*\mathcal Hom(k_U,k_U)=j_*^\circ k_U
\end{align}
And (5) implies (4) because
$$\mathcal F=\mathcal Hom(\mathcal G,\mathcal R^{\vee\vee})=\mathcal Hom(\mathcal G,\mathcal Hom(\mathcal R^{\vee},j_*^\circ k_U))=\mathcal Hom(\mathcal G\otimes\mathcal R^\vee,j_*^\circ k_U)=(\mathcal G\otimes R^\vee)^\vee
$$
Finally, (4) implies (2) because $\mathcal F=\mathcal G^\vee=\mathcal Hom(\mathcal G,j_*^\circ k_U)=j_*\mathcal Hom(\mathcal G|_U,k_U)$ and $\mathcal Hom(\mathcal G|_U,k_U)$ is a local system on $U$.  So all five conditions are equivalent.
\end{proof}

\begin{remark}\label{analogue of normal}
From the proof, one can see that in the complex equidimensional scheme setting, all non-open strata has real codimension at least $2$, so this corresponds to $R_1$ condition in algebraic theory; one also sees that $j_*^\circ k_U$ is $S_2$. So the geometric Serre's theory on complex varieties is similar to the algebraic theory on normal ($S_2$+$R_1$) schemes with a dualizing complex.
\end{remark}

We close this section by generalizing Cohen--Macaulay and Gorenstein in the geometric case.

\begin{definition}
Assume $X$ is a reasonably stratified space. $X$ is called a \textit{Cohen--Macaulay} space if the dualizing complex $\omega_X$ is concentrated in degree $\dim_\mathbb R X$; $X$ is called \textit{Gorenstein} if $\omega_X[-\dim_\mathbb R X]$ is a local system.
\end{definition}
When $X$ is a scheme, to avoid notation conflicts, we will call $X$ \textit{geometric Cohen--Macaulay/Gorenstein} if the respective conditions on constructible sheaves hold. 
\begin{proposition}\label{Cohen--Macaulay and Gorenstein}
For a reasonable stratified real space $X$,
\begin{enumerate}
    \item [(1)] $X$ is Cohen--Macaulay if and only if the constant sheaf $k_X$ is maximal Cohen--Macaulay.
    \item [(2)] $X$ is Gorenstein if and only if $X$ is a  $k$-cohomology manifold, i.e., for any $x\in X$, 
    $$
    H^i(X,X-\{x\},k)=\begin{cases}
    k, \quad i= \dim_\mathbb R X\\
    0, \quad \text{otherwise}
    \end{cases}
    $$ 
    and if $X$ has complex stratification, there is $\omega_X[-\dim_\mathbb R X]=k$.
    \item[(3)] If $X$ is a union of complex varieties, then $X$ is geometric Cohen--Macaulay if and only if $X$ is geometric Gorenstein.
\end{enumerate}
\end{proposition}
\begin{proof}
For (1), it suffices to notice that $\mathbb Dk_X=R\mathcal Hom(k_X,\omega_X)=\omega_X$. For (2), if $X$ is Gorenstein, suppose $\omega_X=\mathcal L[\dim_\mathbb R X]$ and $\mathcal L$ is a local system. Since for open strata $U$, $\omega_X|_U=\omega_U=k_X[\dim_\mathbb R X]$, we see that $\mathcal L$ has rank one. Then, for any $x\in X$, 
$$
\rH^j(i_x^!k_X)=\rH^j(\mathbb D(i_x^*\omega_X))=\rH^j(\mathcal L_x^\vee[-\dim_\mathbb R X])=\rH^{j-\dim_\mathbb R X}(k)=\begin{cases}
    0, \quad j\neq \dim_\mathbb R X\\
    k, \quad j=\dim_\mathbb R X
\end{cases}
$$
Then that $V$ is $k$-cohomology manifold follows from the long exact sequence induced by $R\Gamma i_y^!k_V\rightarrow R\Gamma k_V\rightarrow R\Gamma (j_{V-\{y\}})_*k_{V-\{y\}}\xrightarrow{+1}$. Conversely, if $X$ is a $k$-cohomology manifold, one argue as above that $\rH^j(\mathbb D(i_x^*\omega_X))=\rH^j(i_x^!k_X)=\begin{cases}
    0, \quad j\neq \dim_\mathbb R X\\
    k, \quad j=\dim_\mathbb R X
\end{cases}$, so $\omega_X$ is concentrated in degree $-\dim_\mathbb R X$ with all stalks isomorphic to $k$. To prove $\mathcal L=\omega[-\dim_\mathbb RX]$ is a local system, it suffices to show that no section is supported on strata with dimension $<\dim_\mathbb RX$, so the section in the stalk generates every stalk in a neighborhood. Now, if there is a section $s\in\mathcal F(V)$ such that $s$ is supported on $Z$, then $s_x$ generates all stalks of $\mathcal L_V$ in $Z$, hence $i_Z^*\mathcal L_V=k_Z$. Then, one has a splitting
$${i_Z}_*k_Z\xrightarrow{s}\mathcal L_V\rightarrow {i_Z}_*i_Z^*\mathcal L_V={i_Z}_*k_Z
$$
so if we let $U=V-Z$, there is $\mathcal L_V={i_Z}_*k_Z\oplus {j_{U}}_!\mathcal L$, then
$$
\mathcal End(\mathcal L_V)=\mathcal End({i_Z}_*k_Z\oplus {j_{U}}_!\mathcal L_V)={i_Z}_*k_Z\oplus {j_U}_!\mathcal L_V
$$
However, since $\mathbb Dk_V=\omega_V$, we have $\mathbb D\omega_V=k_V$, so
$$\mathcal End(\mathcal L_V)=\mathcal H^0(R\mathcal Hom(\mathcal L_V,\mathcal L_V))=\mathcal H^0(\mathbb D\omega_V)=k_V
$$
which is non-splitting, a contradiction. So $X$ is Gorenstein. In this case, firstly observe $X$ is equidimensional, since over each open strata, the dualizing complex concentrates on the same degree. Now $\omega_X[-\dim_\mathbb RX]$ is also maximal Cohen--Macaulay since its Verdier dual is $k$, so by \ref{reflexive}, $\omega_X[-\dim_\mathbb R X]=j_*^\circ k_U=k_X$, the last equality is again by \ref{reflexive} and that $k_X$ is at least $S_2$ since it is maximal Cohen--Macaulay.

Finally, for (3), if $k_X$ is maximal Cohen--Macaulay, then $\mathcal G=\omega_X[-\dim_\mathbb RX]$ is a sheaf. While $\mathbb D(\omega_X[-\dim_\mathbb RX])=k_X[\dim_\mathbb R X]$, we see that $\mathcal G$ is also maximal Cohen--Macaulay. Denote $U$ be the union of open strata, by \ref{reflexive}, we see that $\mathcal G=\omega_X[-\dim_\mathbb RX]=j_*^\circ\omega_U[-\dim_\mathbb RU]=j_*^\circ k_U$. Now $k_X$ is maximal Cohen--Macaulay implies $j_*^\circ k_U=k_X$, so $\omega_X[-\dim_\mathbb RX]=\mathcal G=k_X$
\end{proof}

\begin{example}
Suppose $\mathfrak g$ is a complex semisimple Lie algebra. Let $\mathcal N$ be the nilpotent cone in $\mathfrak g$ and suppose $\widetilde{\mathcal N}$ is the Springer resolution. Let $d=\dim_\mathbb C\mathcal N$, according to Springer theory:
$$
\operatorname{IC}_{\mathcal N}(\mathbb Q)
\;\simeq\;
\left(R\pi_*\mathbb Q_{\widetilde{\mathcal N}}[d]\right)^W
\;\simeq\;
\mathbb Q_{\mathcal N}[d].
$$
So $\mathbb D \mathbb Q_\mathcal N[d]=\operatorname{IC}_{\mathcal N}(\mathbb Q^\vee)=\operatorname{IC}_{\mathcal N}(\mathbb Q)=\mathbb Q[d]$, so $\omega_\mathcal N=\mathbb D\mathbb Q_\mathcal N=\mathbb Q_\mathcal N[2d]$. Hence $\mathcal N$ is rational cohomology manifold, hence it is a geometric Gorenstein space in our sense.

However, this only holds for the full nilpotent cones: In \ref{sub nilpotent orbit}, we show that the constant sheaf on nilpotent orbit $\mathcal N_{3,1}$ is not maximal Cohen--Macaulay, so it is not geometrically Cohen--Macaulay. However, $\mathcal N_{3,1}$ is algebraically Gorenstein, this reveals the difference between geometric and algebraic Serre's conditions.
\end{example}

\begin{example}\label{Cn/G}
In \ref{Kleinian surfaces}, we classify the maximal Cohen--Macaulay constructible sheaves on Kleinian surfaces $\mathbb C^2/G$. Indeed, we generalize this to $X=\mathbb C^n/G$ where $G$ is a finite group in $GL(n,\mathbb C)$. Suppose $\Gamma$ is the subgroup of $G$ generated by pseudo-reflections, i.e., those elements $g$ fixing a hyperplane. Notice that quotient by such an element still gives you $\mathbb C^n$, so $X=\mathbb C^n/(G/\Gamma)$ (it is clear that $\Gamma$ is normal, since the conjugates of a pseudo-reflection are pseudo-reflections). Suppose $q$ is the quotient map $\mathbb C^n\rightarrow\mathbb C^n/(G/\Gamma)$, then $q^{-1}(X^{reg})=\mathbb C^n-\bigcup_{g\in G}(\mathbb C^n)^g$. Since $\bigcup_{g\in G}(\mathbb C^n)^g$ is of codimension $\geq 2$, we see that $q^{-1}(X^{reg})$ is simply connected, so it is a universal cover of $X^{reg}$, then $\mathrm{Loc}(X^{reg})\cong \mathrm{Rep}_k(G/\Gamma)$.

Let $k$ be a field such that $\mathrm{char} \,k\nmid |G/\Gamma|$. For any $G/\Gamma$-representation $V$, we define
$$\mathcal F_V=(q_*k_\mathbb C^n\otimes V)^{G/\Gamma}
$$
The fiber of $V$ at a point $x=q(z)$ is $V^{((G/\Gamma)_z}$, so the generic fiber is $V$. Therefore $\mathcal F_V|_{X^{reg}}$ is the sheaf with respect to the $\pi_1(X^{reg})$-representation $V$. Now since $\mathrm{char} \,k\nmid |G/\Gamma|$, taking $G$-invariant is exact, so
$$\mathbb D\mathcal F_V=\mathbb D(q_*k_\mathbb C^n\otimes V)^{G/\Gamma}=(\mathbb Dq_*k_\mathbb C^n\otimes V^\vee)^{G/\Gamma}=(q_*\mathbb Dk_\mathbb C^n\otimes V^\vee)^{G/\Gamma}=(q_*k_\mathbb C^n[2n]\otimes V^\vee)^{G/\Gamma}=\mathcal F_{V^\vee}[2n]
$$
So $\mathcal F_V$ is maximal Cohen--Macaulay in our sense. hence $\mathcal F=(j_{X^{reg}})_*(\mathcal F_V|_{X^{reg}})$ and we conclude that
$$\mathrm{MCM}(X)=\mathrm{Loc}(X^{reg})\cong \mathrm{Rep}_k(G/\Gamma)
$$

\end{example}
\subsection{Relation with intermediate extension and perversity}\label{geometric Serre's conditions and perversity}
In this subsection, suppose $X$ is a complex variety with Whitney stratification $X=\bigcup_s X_s$, and let $j:U\hookrightarrow X$ be the union of open strata. In algebraic theory, we have the relationships between $S_1$,$S_2$, and torsion-freeness, reflexivity, respectively, but we lack a good explanation for higher Serre's conditions $S_r \,(r\geq 3)$. In this subsection, we will reveal the relationship between the geometric Serre's condition and perversity.

Let $n_\mathbb R=\dim_\mathbb R X$. A first glimpse gives you a feeling that the Serre's condition $S_d$ is an analoguey of perverse condition removing a restriction at cohomology degree $d-n$, since one can write the condition for a constructible sheaf $\mathcal F$ in the following way:
$$\dim_\mathbb R\mathrm{Supp}\mathcal H^i(\mathbb D\mathcal F[d])\leq -i,\quad i\neq d-n_\mathbb R
$$
This illusion could be easily justified by realizing that the dimension here in the geometric Serre's conditions is real dimension, while the positive perverse structure is defined using complex dimension: ${}^pD_c^{\geq0}(X)=\{\mathcal F\in D^b_c(X)|\dim_\mathbb C\mathrm{Supp} \mathcal H^i(\mathbb D\mathcal F)\leq -i\}$. In order to correctly compare these two conditions, we normalize $\mathcal F$ by shifting it by $n_\mathbb C$ and rewrite our Serre's condition in complex dimension:
$$S_d:\ \dim_\mathbb C\mathrm{Supp}\mathcal H^i(\mathbb D\mathcal F[n_\mathbb C])\leq\frac{n_\mathbb C-i-d}{2},\quad i\neq  -n_\mathbb C
$$
Using this, we immediately get the following result.
\begin{proposition}
If $\mathcal F$ is $S_{n_\mathbb C}$, then $\mathcal F[n_\mathbb C]$ is perverse.
\end{proposition}
\begin{proof}
$S_{n_\mathbb C}$ means $\dim_\mathbb C\mathrm{Supp}\mathcal H^i(\mathbb D\mathcal F[n_\mathbb C])\leq-\frac{i}{2},\quad i\neq  -n_\mathbb C$. For $i\leq0,\  i\neq -n_\mathbb C$, there is $-\frac{i}{2}\leq -i$, so the perverse conditions are satisfied. The perverse condition for $i=-n_\mathbb C$ trivially holds. For  $i>0$, both $S_{n_\mathbb C}$ and the perverse conditions are the same as $\mathcal H^i(\mathbb D\mathcal F[n_C])=0$, so when $\mathcal F$ is $S_{n_\mathbb C}$, $\mathcal F[n_\mathbb C]$ is perverse.
\end{proof}

Next, we want to find more connections between Serre's conditions and perversity. We need the following two characterizations of perversity and intermediate extension.
\begin{lemma}\label{perverse lemma}[\cite{Achar2021PerverseSheaves} Lemma 3.1.5]
Let $X$ be a complex variety, and let $(X_s)_{s\in\mathscr S}$ be a stratification. For each $s\in\mathscr S$, let $j_s:X_s\hookrightarrow X$ be the inclusion map, and let $\mathcal F\in D_c^b(X,k)$.
\begin{enumerate}
    \item Suppose $\mathcal F$ is constructible with respect to $\mathscr S$. We have $\mathcal F\in {}^pD_c^b(X,k)^{\leq 0}$ if and only if $j_s^*\mathcal F\in D_{\mathrm{locf}}^b(X_s,k)^{\leq -\dim X_s}$ for all $s\in\mathscr S$.

    \item Suppose $\mathcal F$ is constructible with respect to $\mathscr S$. We have $\mathcal F\in {}^pD_c^b(X,k)^{\geq 0}$ if and only if $j_s^!\mathcal F\in D_{\mathrm{locf}}^b(X_s,k)^{\geq -\dim X_s}$ for all $s\in\mathscr S$.
\end{enumerate}
\end{lemma}
\begin{lemma}\label{intermediate extension lemma}[\cite{Achar2021PerverseSheaves} Lemma 3.3.4]
Let $\mathcal F$ be a perverse sheaf on $X$. Let $h:Y\hookrightarrow X$ be the inclusion map of a locally closed subvariety, and let $\mathcal F'$ be a perverse sheaf on $Y$. Let $Z=\overline{Y}\setminus Y$, and let $i:Z\hookrightarrow X$ be the inclusion map. The following conditions are equivalent:
\begin{enumerate}
    \item $\mathcal F\cong h_{!*}\mathcal F'$.
    \item The support of $\mathcal F$ is $\overline{Y}$. Moreover, $\mathcal F|_Y\cong\mathcal F'$, and we have
    $$
    i^!\mathcal F\in {}^pD_c^b(Z,k)^{\geq 1}
    \qquad\text{and}\qquad
    i^*\mathcal F\in {}^pD_c^b(Z,k)^{\leq -1}.
    $$
\end{enumerate}
\end{lemma}

With these lemmas and \ref{Generalized duality}, we can tell the following important relationship between Serre's conditions and perversity.
\begin{theorem}\label{Serre's condition and perversity}
Let $X(r)=\bigcup_{\mathrm{codim_\mathbb C}X_s\leq r}X_s$ to be the union of strata of complex codimension $\leq r$. Let $j_r:U\hookrightarrow X(r)$ be the open inclusion of the union of open strata into $X(r)$. Then if $\mathcal F$ is $S_r$, we have
\begin{enumerate}
    \item [(1)] $\mathcal F[n_\mathbb C]|_{X(r)}$ is perverse on $X(r)$.
    \item [(2)] $\mathcal F[n_\mathbb C]$ is the intermediate extension of $\mathcal F[n_\mathbb C]|_U$ on $X(r-1)$, i.e.,
    $$
    \mathcal F[n_\mathbb C]=(j_{r-1})_{!*}\mathcal F|_U[n_\mathbb C]=\mathrm{IC}(X(r),\mathcal F|_U)
    $$
\end{enumerate}
\end{theorem}
\begin{proof}
Suppose $\mathcal F$ is $S_r$, by proposition \ref{Generalized duality}, we have $i_{X_s}^!\mathcal F \in D^{\geq\min\{r,\mathrm{codim}_\mathbb R X_s\}}_{\mathrm{Loc}}(X_s)$, so 
$$i_{X_s}^!\mathcal F[n_\mathbb C] \in D^{\geq\min\{r-n_\mathbb C,\,2\mathrm{codim}_\mathbb C X_s-n_\mathbb C\}}_{\mathrm{Loc}}(X_s)
$$
Now for non-open stratum $X_s$, $2\mathrm{codim}_\mathbb C X_s-n_\mathbb C=\mathrm{codim}_\mathbb C X_s-\dim_\mathbb C X_s\geq-\dim_\mathbb C X_s+1$. If $\mathrm{codim_\mathbb C}\leq r$, we have $r-n_\mathbb C\geq -\dim_\mathbb C X_s$, so $i_{X_s}^!\mathcal F[n_\mathbb C] \in D^{\geq -\dim X_s}$. Then by lemma \ref{perverse lemma}, This implies $\mathcal F$ is perverse on $X(r)$. Similarly, if  $\mathrm{codim_\mathbb C}\leq r-1$, we have $r-n_\mathbb C \geq -\dim_\mathbb C X_s+1$, so $i_{X_s}^!\mathcal F[n_\mathbb C] \in D^{\geq -\dim X_s+1}$. Then combined with \ref{perverse lemma} and \ref{intermediate extension lemma}, we see that $\mathcal F$ is the intermediate extension of $\mathcal F|_U$. 
\end{proof}

This theorem is a very good explanation of $S_r$ conditions for $r\leq \dim_\mathbb C X+1$. Two immediate interesting corollaries appear here:
\begin{corollary}\label{middle Serre's condition and perverse}
If $\mathcal F$ is $S_{\dim_\mathbb C X}$, then $\mathcal F[n_\mathbb C]$ is perverse.
\end{corollary}
\begin{corollary}\label{middle Serre's condition and intersection complex}
If $\mathcal F$ is $S_{\dim_\mathbb C X+1}$, then $\mathcal F[n_\mathbb C]=\mathrm{IC}(X,\mathcal F|_U[n_\mathbb C])$.
\end{corollary}

For maximal Cohen--Macaulay condition, we have the following theorem.
\begin{theorem}\label{MCM and intersection complex}
Suppose $\mathcal F$ is a constructible sheaf and $\mathcal F|_U=\mathcal L$ is a local system on $U$. Then $\mathcal F$ is maximal Cohen--Macaulay if and only if $\mathcal F=j^\circ_*\mathcal L$ and
$$j^\circ_*\mathcal L[n_\mathbb C]=\mathrm{IC}(X,\mathcal L),\quad j^\circ_*\mathcal L^\vee[n_\mathbb C]=\mathrm{IC}(X,\mathcal L^\vee)
$$
\end{theorem}
\begin{proof}
If $\mathcal F$ is maximal Cohen--Macaulay, by \ref{reflexive}, we know that $\mathcal F=j_*^\circ$. By \ref{Serre's condition and perversity}, we know that $j^\circ_*\mathcal L[n_\mathbb C]=\mathrm{IC}(X,\mathcal L)$. Now suppose $\mathbb D\mathcal F=\mathcal G[\dim_\mathbb RX]$. Then $\mathbb D\mathcal G=\mathcal F[\dim_\mathbb R X]$ so $\mathcal G$ is also maximal Cohen--Macaulay. Since $\mathcal G|_U=\mathcal L^\vee$, we know from  \ref{Serre's condition and perversity} that $\quad j^\circ_*\mathcal L^\vee[n_\mathbb C]=\mathrm{IC}(X,\mathcal L^\vee)$.

Conversely, If $\mathcal F=j_*^\circ\mathcal L$ and the two conditions hold. By \cite{Achar2021PerverseSheaves} Lemma 3.3.13, we know that $\mathbb D(\mathrm{IC}(X,\mathcal L))=\mathrm{IC}(X,\mathcal L^\vee)$, so $\mathbb Dj_*^\circ\mathcal L[n_\mathbb C]=j_*^\circ\mathcal L^\vee[n_\mathbb C]$, hence $\mathbb D\mathcal F=j_*^\circ\mathcal L^\vee[n_\mathbb R]$, so $\mathcal F$ is maximal Cohen--Macaulay.
\end{proof}

Perversity can conversely imply some of Serre's conditions, as follows.
\begin{proposition}\label{Perverse implies Serre's conditions}
Suppose $c=\mathrm{codim}_\mathbb C(X-U)$, for a constructible sheaf $\mathcal F$ on $X$, we have
\begin{enumerate}
    \item [(1)] If $\mathcal F[n_\mathbb C]$ is perverse, then $\mathcal F$ is at least $S_{c}$
    \item [(2)] If $\mathcal F[n_\mathbb C]=\mathrm{IC}(X,\mathcal L)$ for some $\mathcal L$ on $U$, then $\mathcal F$ is at least $S_{\max\{3,c+1\}}$.
\end{enumerate}
\end{proposition}
\begin{proof}
For (1), if $\mathcal F[n]$ is perverse, then by \ref{perverse lemma}, for any strata $S$, $i_S^!\mathcal F\in D^{\geq \mathrm{codim}_\mathbb CS}\subset D^{\geq c}$. The statement then follows from \ref{Generalized duality}. For (2), if $\mathcal F[n]$ is an intersection complex, by \ref{intermediate extension lemma}, $i_S^!\mathcal F\in D^{\geq \mathrm{codim}_\mathbb CS+1}\subset D^{\geq c+1}$, so by \ref{Generalized duality} $\mathcal F$ is $S_{c+1}$. On the other hand, since $\mathrm{codim}_\mathbb CS+1\geq\min\{3,2\mathrm{codim}_\mathbb CS\}$, again by \ref{Generalized duality} $\mathcal F$ is $S_3$, so $\mathcal F$ is at least $S_{\max\{3,c+1\}}$.
\end{proof}

In particular, combining this and corollary \ref{middle Serre's condition and perverse} and \ref{middle Serre's condition and intersection complex}, we have the following result for isolated singularities.
\begin{corollary}\label{On isolated singularities}
When the stratification is $X=U\cup\{p\}$, we have
\begin{enumerate}
\item [(1)] $\mathcal F$ is $S_{n_\mathbb C}$ if and only if $\mathcal F[n_\mathbb C]$ is perverse. 
\item [(1)] $\mathcal F$ is $S_{n_\mathbb C+1}$ if and only if $\mathcal F[n_\mathbb C]=\mathrm{IC}(X,\mathcal F|_U)$.
\end{enumerate}
\end{corollary}

Making use of theorem \ref{MCM and intersection complex}, we have the following result for self-dual and finite-monodromy local systems on $U$.
\begin{corollary}\label{self dual and finite monodromy}
If $k=\mathbb C$, for a local system $\mathcal L$ on $U$ such that either of the following conditions hold:
\begin{enumerate}
    \item [(1)] $\mathcal L=\mathcal L^\vee$
    \item [(2)] The image of $\pi_1(U)$ is $GL(\mathcal L_x,\mathbb C)$ is a finite group.
\end{enumerate}
Then we have
$$j_*\mathcal L \ \text{is}\ S_{n_\mathbb C+1}\ \Longleftrightarrow\ j_*\mathcal L=\mathrm{IC}(X,\mathcal L)[-n_\mathbb C]\ \Longleftrightarrow\ j_*\mathcal L \ \text{is maximal Cohen--Macaulay}
$$
\end{corollary}
\begin{proof}
If (1) holds, the conclusion follows from corollary \ref{middle Serre's condition and intersection complex} and theorem \ref{MCM and intersection complex}. If (2) holds, let $G$ to be the image of $\pi_1(U)$ and start with any positive-definite Hermitian form \(h_0\) on \(V\). Do average:
\[
h(v,w)=\frac1{|G|}\sum_{g\in G}h_0(gv,gw).
\]
we get a $G$-invariant Hermitian form on $V$. Then $\mathcal L^\vee=\overline{\mathcal L}$, the complex conjugation of $\mathcal L$. Then the conclusion again follows from corollary \ref{middle Serre's condition and intersection complex} and theorem \ref{MCM and intersection complex}.
\end{proof}
The following example shows that not all constructible sheaves which are shifted intersection complex are maximal Cohen-Macaulay.
\begin{example}\label{intersection complex but not MCM}
Suppose $X=\mathbb C^2$ with stratification $X=\{0\}\cup A_x\cup A_y\cup U$ where $A_x=\{(x,0)|x\neq 0\}$ and $A_x=\{(0,y)|y\neq 0\}$. Now $\pi_1(U)\cong\mathbb Z_x\oplus \mathbb Z_y$. Consider a $3$-dimensional vector space $V=\mathbb C\langle u,v,w\rangle$ and two commuting linear operations
$$F_x((u,v,w))=(u,v+u,w),\quad F_y((u,v,w))=(u,v,w+u)
$$
This endows $V$ a $\pi_1(U)$-representation structure, hence gives a local system $\mathcal L$ on $U$. Let $\mathcal F=j_* \mathcal L$. By the computation in \ref{punctured disck}, we see that $\mathcal F$ is maximal Cohen-Macaulay along the axes $A_x$ and $A_y$. Let $U_x=\mathbb C^2-A_x$, $U_y=\mathbb C^2-A_y$ and $U_0=U_x\cup U_y$. At $0$, the Mayer-Vietoris sequence gives
$$
...\longrightarrow H^1(U_0,\mathcal F)
\longrightarrow H^1(U_x,\mathcal F)\oplus H^1(U_y,\mathcal F)
\xlongrightarrow{\delta} H^1(U,\mathcal L),
\longrightarrow H^2(U_0,\mathcal F)
\longrightarrow 0.
$$
Now, using Mayer-Vietoris sequence tricks again, we compute that $H^1(U_x,\mathcal F)=\mathrm{coker}_{\ker (F_y-id)}(F_x-id)=\mathbb C\langle v\rangle$, $H^1(U_y,\mathcal F)=\mathrm{coker}_{\ker (F_x-id)}(F_y-id)=\mathbb C\langle w\rangle$ and $H^1(U,\mathcal L)=\mathbb C\langle \alpha,\beta,\gamma\rangle$. One can also figure out that the map $\delta$ is $\delta(av,bu)=a\alpha-b\beta$. So, $H^1(U_0,\mathcal F)=0$, $H^2(U_0,\mathcal F)=\mathbb C$. By the localizing triangle $i_0^!\mathcal F\rightarrow i_0^*\mathcal F\rightarrow i_0^*j_*\mathcal L$, we see that $H^3(i_0^!\mathcal F)=H^2(U_0,\mathcal F)\neq 0$. This means that $\mathcal F$ is not maximal Cohen-Macaulay. However, $H^2(i_0^!\mathcal F)=H^1(U_0,\mathcal F)=0$ and $H^{0,1}(i_0^!\mathcal F)$ since $\mathcal F$ is $S_2$, we see that $\mathcal F$ is $S_3$, which implies by \ref{middle Serre's condition and intersection complex} that $\mathcal F[2]=\mathrm{IC(\mathbb C^2,\mathcal L)}$. 
\end{example}

\subsection{The geometric Auslander--Reiten Conjecture}\label{Auslander--Reiten Conjecture}
With a Verdier-type category and our Serre's condition construction, one can ask many questions that arise in commutative algebra and classical algebraic geometry. Let's give one example in this subsection. The Auslander--Reiten conjecture is a long-standing open problem in commutative algebra, it suggests for any noetherian local ring $R$ and an $R$-module $M$, if $\mathrm{Ext}^{>0}(M,M\oplus R)=0$, then $M$ is projective. In \cite{KimuraARNormal}, the Auslander--Reiten conjecture for normal rings has been proved. Since by remark \ref{algebraic Serre's conditions}, our geometric setting on complex varieties is an analogue of algebraic normal schemes with a dualizing complex, the Auslander--Reiten conjecture should be true. Here we give a proof of it, and see how it feeds back to the algebraic Auslander--Reiten conjecture.
\begin{theorem}[Geometric Auslander--Reiten Conjecture]\label{ARC}
Suppose $X$ is a equidimensional complex variety with Whitney stratification. Fix a field $k$ of characteristic $0$ and a constructible sheaf $\mathcal F$ on $X$ over $k$, if $\mathcal F$ is geometrically maximal Cohen--Macaulay, and $\mathcal Ext^{>0}(\mathcal F,\mathcal F)=0$, then $\mathcal F$ is a local system. 
\end{theorem}
\begin{proof}
Let $n=\dim_\mathbb RX$. Since $\mathcal F$ is maximal Cohen--Macaulay, we can suppose $\mathbb D\mathcal F=\mathcal G[n]$ for some constructible sheaf $\mathcal G$. Now it follows that $\mathbb D\mathcal G=\mathcal F[n]$, so $\mathcal G$ is also maximal Cohen--Macaulay. Now $\mathcal Ext^{>0}(\mathcal F,\mathcal F)=0$, so $R\mathcal Hom(\mathcal F,\mathcal F)=\mathcal End(\mathcal F)$. On the other side:
$$R\mathcal Hom(\mathcal F,\mathcal F)=\mathbb D(\mathcal F\otimes_{k_X}^L \mathbb D \mathcal F)=\mathbb D(\mathcal F\otimes\mathcal G[n])=\mathbb D(\mathcal F\otimes_{k_X} G)[-n]
$$
so $\mathbb D(\mathcal F\otimes_{k_X} \mathcal G)=\mathcal End(\mathcal F)[n]$. Then we see that both $\mathcal F\otimes_{k_X} \mathcal G$ and $\mathcal End(\mathcal F)$ are maximal Cohen--Macaulay. 

Now suppose $j:U\hookrightarrow X$ is the union of open strata and let $\mathcal F|_U=\mathcal L$ where $\mathcal L$ is a local system on $U$. Let $A=j_*^\circ k_U$. By \ref{reflexive}, $A$ is reflexive. Now $\mathcal F$, $\mathcal G$ and $\mathcal End(\mathcal F)$ are endowed a $A$-module structure induced by the pushforward of action on $U$. one checks by pushing forward the corresponding relations on $U$ that the action is tensor balanced, i.e., for any $x\in X$, $a\in A_x$, $m\in\mathcal F_x$ and $n\in\mathcal G_x$, one has $am\otimes n=m\otimes an$. Now take $f\in \mathcal G_x'$ such that $f(n)=1$, evaluate both sides by $id_{\mathcal F_x}\otimes f$, one sees that $am=f(an)m$, so $a$ acts on $\mathcal F_x$ as a scalar $\lambda$. Similarly, $a$ acts on $\mathcal G_x$ by scalar $\lambda$.

Since $\mathcal F\otimes_{k_X} \mathcal G$ and $\mathcal End(\mathcal F)$ are maximal Cohen--Macaulay, again by \ref{reflexive}, $\mathcal F\otimes_{k_X} \mathcal G=j_*^\circ(\mathcal L\otimes_{k_X}\mathcal L^\vee)$, $\mathcal End(\mathcal F)=j_*^\circ\mathcal End(\mathcal L)$, and the evaluation isomorphism $\mathcal L\otimes \mathcal L^\vee\xrightarrow{\sim}\mathcal End(\mathcal L)$ pushforwards to an isomorphism
$$\theta:\mathcal F\otimes_{k_X}\mathcal G\xrightarrow{\sim}\mathcal End(\mathcal F)
$$
This isomorphism is $A$-balanced, i.e., $a\theta(m\otimes n)=\theta(am\otimes n)=\theta(m\otimes an)$. Those are also relations pushed forward from $U$. Now by discussion above, for any $x\in X$, $a$ acts on $\mathcal End(\mathcal F)_x$ as a scalar $\lambda$. On the other hand, the map multiplication by scalar $k_U\rightarrow\mathcal End(\mathcal L)$ pushforwards to $A\rightarrow\mathcal End( \mathcal F):a\rightarrow a\cdot id_\mathcal F$ which is generically injective. Now $A$ is $S_2$, a fortiori, it is torsion free, which shows that this map is injective. However, our discussion above shows that $(a-\lambda.1)\cdot id_\mathcal F=0$, so $a=\lambda.1$. Since $a$ is any arbitrary element in any stalk, this shows that $A=k_X$.

Now the splitting $k_U\xrightarrow{a\mapsto a\cdot id} \mathcal End(\mathcal L)\xrightarrow{tr/rank} k_U$ pushforwards to to $k_X\rightarrow\mathcal End(\mathcal F)\rightarrow k_X$, so $k_X$ is a summand of a maximal Cohen--Macaulay sheaf, hence also maximal Cohen--Macaulay by \ref{Serre's conditions in extensions}. Then $X$ is geometrically Cohen--Macaulay. Since $X$ is a complex space, Proposition \ref{Cohen--Macaulay and Gorenstein} tells us that $X$ is indeed geometrically Gorenstein, so $\mathbb Dk_X=k_X[n]$. Then we see that $\mathcal G=\mathcal F^\vee$, the map $\theta$ is the usual evaluation map $\mathcal F\otimes_{k_X}\mathcal F^\vee\rightarrow\mathcal End(\mathcal F)$. It is an isomorphism by the discussion above, a fortiori, it is surjective. Now consider the section $id_\mathcal F$, by the definition of a surjective sheaf map, for any $x\in X$, there is a neighborhood $V$ such that $id_F$ can be represented as $\theta(\sum_{i=1}^Nm_i\otimes\phi_i)$. Now consider the splitting
$$\mathcal F|_V\xrightarrow{(\phi_1,\phi_2,...\phi_N)} k_V^N\xrightarrow{(m_1,m_2,...,m_N)}\mathcal F|_V
$$
we see that the $\mathcal F|_V$ is a direct summand of $k_V^N$. Take a connected neighborhood $V'$ in $V$, we see that $\mathcal F|_{V'}$ is constant. Hence, $\mathcal F$ is a local system.
\end{proof}

Let's see how this method feeds back to the classical Auslander--Reiten conjecture. Let's consider the Gorenstein case: Suppose $R$ is a Gorenstein ring, either global or local, and $M$ is an $R$-module with $\mathrm{Ext}^{>0}_R(M,M\oplus R)=0$. Let $n=\dim R$, we translate this condition into
$$\mathbb DM=M^\vee[n],\quad  R\mathrm{Hom}_R(M,M)=\mathrm{End}_R(M)
$$
Then we immediately know that $M^\vee$ is also maximal Cohen--Macaulay since one can apply dual again. Now we also have the formula in derived category:
$$R\mathrm{Hom}_R(M,M)=\mathbb D(M\otimes^L\mathbb DM)=\mathbb D(M\otimes^L M^\vee)[-n]
$$
So $M\otimes^L M^\vee=\mathbb D(R\mathrm{Hom}_R(M,M))[-n]=\mathbb D(\mathrm{End}_R(M))[-n]$. It seems that the derived tensor product will cause trouble, but in fact not, since
$$M\otimes^L M^\vee\in D^{\leq 0}(\mathfrak Mod(R)),\quad \mathbb D(\mathrm{End}_R(M))[-n]=R\mathrm{Hom}_R( \mathrm{End}_R(M), R)\in D^{\geq0}(\mathfrak Mod(R))
$$
so actually $M\otimes^L M^\vee=\mathbb D(\mathrm{End}_R(M))[-n]\in\mathfrak Mod(R)$, which tells us $M\otimes ^LM^\vee=M\otimes M^\vee$ and $\mathbb D(\mathrm{End}_R(M))=M\otimes M^\vee[n]$. Hence  both $M\otimes M^\vee$ and $\mathrm{End}_R(M)$ are maximal Cohen--Macaulay.

Now if the Auslander--Reiten conjecture holds for dimension $1$ Gorenstein rings, then $M\otimes M^\vee$ and $\mathrm{End}_R(M)$ are locally free in codimension $1$, so their singular locus on $\mathrm{Spec} R$ are both of codimension $\geq 2$, and there is a open set $U$ such that both $M\otimes M^\vee$ and $\mathrm{End}_R(M)$ are locally free on $U$ and $\mathrm{Spec} R-U$ is of codimension $\geq 2$. Then one can again play the game, the natural isomorphism $M_U\otimes M_U^\vee\xrightarrow{\sim}\mathrm{End}_R(M_U)$ pushes forward to an isomorphism
$$\theta: M\otimes M^\vee\xrightarrow{\sim}\mathrm{End}_R(M)
$$
so identity is in the pre-image, namely $\theta(\sum_{i=1}^Nm_i\otimes\phi_i)=id_M$. Then
$$M\xrightarrow{(\phi_1,\phi_2,...\phi_N)} R^N\xrightarrow{(m_1,m_2,...,m_N)}\mathcal M
$$
splits and hence $M$ is projective. Thus we conclude that
\begin{proposition}
The Auslander--Reiten conjecture for Gorenstein rings holds if and only if it holds for all $1$-dimensional Gorenstein rings.
\end{proposition}
In fact, by Proposition 2.2 of \cite{Kumashiro2023}, one can reduce the Gorenstein Auslander--Reiten conjecture to $0$-dimensional case. One interesting formulation of this conjecture is by using singularity category. The singularity category is defined by 
$$\mathrm{Sing}(R)=D^b(\mathfrak Mod(R))/\mathrm{Perf} =\underline{ \mathrm{MCM}}(R)
$$
where $\underline{ \mathrm{MCM}}(R)$ is the stable category of maximal Cohen--Macaulay $R$-modules. The equality on the right is by \cite{Buchweitz1986}. This is a triangulated category. The suspension functor $\Omega$ is defined by take the first syzygy of the free resolution $0\rightarrow\Omega M\rightarrow R^n\rightarrow M\rightarrow 0$. In singular category, $[R^n]=0$, so $\Omega M=M[1]$. Now the Auslander--Reiten conjecture can be translated in the following form
$$\langle M,\Omega^i M\rangle=0 \ \text{for} \ i\neq 0\  \Rightarrow\ \langle M,M\rangle=0
$$
One can also ask this question for other categories, while some answers could be negative. For example, one can ask this in the homotopy category of spectra, the counter example is $H\mathbb Q$ since $\mathrm{map}(H\mathbb Q,H\mathbb Q)=H\mathbb Q$ and therefore $\langle H\mathbb Q,H\mathbb Q\rangle=\mathbb Q$ and $\langle H\mathbb Q,H\mathbb Q[i]\rangle=0$ if $i\neq 0$. However, in our reduction, we can only consider the Artinian Gorenstein local rings, all modules admits a finite filtration by $\kappa(R)$, the residue field of $R$, so the singularity category contains more finiteness and structure. Understanding those structure could be a potential way to prove the conjecture.

\bibliographystyle{amsalpha}
\bibliography{references}
\end{document}